\documentclass[10pt]{amsart}
\usepackage{amsthm}
\usepackage{graphicx}
\usepackage{amssymb}
\usepackage{color}
\usepackage{tikz}
\usepackage{amsmath}
\usepackage{hyperref}
\usepackage{float}
\usepackage{amsthm}
\usepackage{enumerate}

\DeclareMathAlphabet\mathbfcal{OMS}{cmsy}{b}{n}

\numberwithin{equation}{section}

\newtheorem{theorem}{Theorem}[section]
\newtheorem{definition}[theorem]{Definition}

\newtheorem{lemma}[theorem]{Lemma}

\newtheorem{proposition}[theorem]{Proposition}
\newtheorem{corollary}[theorem]{Corollary}
\newtheorem{remark}[theorem]{Remark}

\newtheoremstyle{named}{}{}{\itshape}{}{\bfseries}{.}{.5em}{\thmnote{#3\ }#1}
\theoremstyle{named}
\newtheorem*{namedtheorem}{Theorem}

 \newcommand{\RR}{\mathbb{R}}
\newcommand{\QQ}{\mathbb{Q}} \newcommand{\CC}{{\mathbb C}}
 \newcommand{\NN}{{\mathbb N}}

\newcommand{\cB}{\mathcal{B}}

\newcommand{\cE}{\mathcal{E}}

\newcommand{\cH}{\mathcal{H}}

\newcommand{\cJ}{\mathcal{J}}

\newcommand{\cL}{\mathcal{L}}

\newcommand{\cP}{\mathcal{P}}

\newcommand{\Ric}{\operatorname{Ric}}
\newcommand{\vol}{\operatorname{vol}}
\newcommand{\ord}{\mathrm{ord}}

\newcommand{\ddc}{dd^c}

\title{A quantization proof of the uniform Yau--Tian--donaldson conjecture}

\usepackage{hyperref}
\hypersetup{colorlinks,linkcolor={blue},citecolor={red}}

\begin{document}

\author[K. Zhang]{Kewei Zhang}
\address{Laboratory of Mathematics and Complex Systems, School of Mathematical Sciences, Beijing Normal University, Beijing, 100875, P. R. China}
\email{kwzhang@bnu.edu.cn}

\maketitle

\begin{abstract}
Using quantization techniques, we show that the $\delta$-invariant of Fujita--Odaka coincides with the optimal exponent in certain Moser--Trudinger type inequality. Consequently we obtain a uniform Yau--Tian--Donaldson theorem for the existence of twisted K\"ahler--Einstein metrics with arbitrary polarizations. Our approach mainly uses pluripotential theory, which does not involve Cheeger--Colding--Tian theory or the non-Archimedean language. A new computable criterion for the existence of constant scalar curvature K\"ahler metrics is also given.
\end{abstract}

\tableofcontents

\section{Introduction}

A fundamental problem in K\"ahler geometry is to find canonical metrics on a given manifold. A problem of this sort is often called the Yau--Tian--Donaldson (YTD) conjecture, which predicts that the existence of canonical metrics is equivalent to certain algebro-geometric stability notion. This article, as a continuation of the author's recent joint work with Rubinstein--Tian \cite{RTZ20}, is mainly concerned with the existence of twisted K\"ahler--Einstein (tKE) metrics on projective manifolds. 
We will present a short quantization proof of a uniform version of the YTD conjecture, by directly relating Fujita--Odaka's $\delta$-invariant \cite{FO18} (that characterizes unform Ding stability \cite{BJ17,BoJ18}) to the existence of tKE metrics. 

The key ingredient in our approach is the analytic $\delta$-invariant defined as the optimal exponent of certain Moser--Trudinger inequality, which we denote by $\delta^A$. This analytic threshold characterizes the coercivity of Ding functionals and hence governs the existence of tKE metrics. 
In the prequel \cite{RTZ20} we set up a quantization approach whose goal is to show that $\delta$ and $\delta^A$ are actually equal, a conjecture previously made by the author in \cite{Zha20}. If this works out then one would have a new proof
for the uniform YTD conjecture. Although this goal was not achieved in \cite{RTZ20}, we were able to prove a quantized version saying that $\delta_m=\delta^A_m$ indeed holds at each level $m$, so that $\delta_m$ characterizes the existence of certain balanced metrics in the $m$-th Bergman space. In the view of Donaldson's quantization framework \cite{Don01}, this makes our conjectural picture about $\delta$ and $\delta^A$ even more promising.

In this article we completely solve our conjecture. Our result can be viewed as an analogue of Demailly's result \cite[Appendix]{CS08} (see also Shi \cite{Shi10}) on the algebraic interpretation of Tian's $\alpha$-invariant, the proof of which actually greatly influenced this article and its prequel \cite{RTZ20}.

\begin{namedtheorem}[Main]
The equality $\delta(L)=\delta^A(L)$ holds for any ample line bundle $L$.
\end{namedtheorem}

Consequently we obtain a new proof of the uniform YTD conjecture, in a much simpler fashion than the other known approaches in the literature.  More precisely, our approach only uses the following analytic ingredients:
\begin{itemize}
	\item Tian's seminal work \cite{Tian89} on the asymptotics of Bergman kernels (see also Bouche \cite{Bou90});
	\item the lower semi-continuity result of Demailly--Koll\'ar \cite{DK01};
	\item the existence of geodesics in the space of K\"ahler metrics going back to Chen \cite{Chen00};
	\item the variational approach of Berman, Boucksom, Eyssidieux, Guedj and Zeriahi \cite{BBGZ,BBEGZ19};
	\item a quantized maximum principle due to Berndtsson \cite{Bern09}.
\end{itemize}
While on the algebraic side, we only need
\begin{itemize}
	\item Fujita--Odaka's basis divisor characterization of $\delta_m$ \cite{FO18};
	\item Blum--Jonsson's valuative definition of $\delta$ \cite{BJ17}.
\end{itemize}

When the underlying manifold is Fano,
a special case of our main theorem  has essentially been obtained by Berman--Boucksom--Jonsson \cite{BBJ18} (see also \cite[Appendix]{CRZ19} and \cite[Corollary 3.10]{Zha20}), which says that $\min\{s,\delta\}=\min\{s,\delta^A\}=\ $the greatest Ricci lower bound, where $s$ denotes the nef threshold. Note that the approach in \cite{BBJ18} crucially relies on the convexity of twisted K-energy and the compactness of weak geodesic rays, which unfortunately cannot directly yield $\delta=\delta^A$ when these thresholds surpass $s$.
In contrast, 
our quantization argument mainly takes place in the finite dimensional Bergman space without involving the convexity of Ding or Mabuchi functionals. Hence as a consequence, we can treat arbitrary (even irrational!) polarizations and establish the very much desired equality $\delta=\delta^A$. Somewhat surprisingly, our approach not only yields stronger results, but in fact comes with a quite short proof. \footnote{However we should emphasize that the non-Archimedean formalism in \cite{BBJ18} indeed plays a key role when it comes to the cscK problem; see e.g. \cite{Li20-cscK} for some recent breakthrough.}
Note that our methods extend easily to the case of klt currents as treated in \cite{BBJ18} (which we will indeed adopt in what follows), and more generally also to the coupled soliton case considered in \cite{RTZ20}. Our work even has applications in finding constant scalar curvature K\"ahler (cscK) metrics, since we will give a new computable criterion for the coercivity of the K-energy.

\textbf{Organization.} The rest of this article is organized as follows. We will fix our setup and notation, and state more precisely our main results in Section \ref{sec:setup-results}. In Section \ref{sec:existence} we elaborate on how $\delta^A$ is related to the existence of canonical metrics. Then in Section \ref{sec:quantization} we recall some necessary quantization techniques on the Bergman space and prove the key estimate, Proposition \ref{prop:E<1-e-E-m}. Finally, our main results, Theorems \ref{thm:delta-delta-A}, \ref{thm:YTD} and \ref{thm:cscK}, are proved in Section \ref{sec:proof}.

\textbf{Acknowledgments.}
The author is grateful to Gang Tian for inspiring conversations during this project. He also thanks Chi Li, Yanir Rubinstein, Yalong Shi, Feng Wang, Mingchen Xia and Xiaohua Zhu for reading the first draft and for valuable comments. Special thanks go to Bo Berndtsson and Robert Berman for letting me know of an alternative proof of Proposition \ref{prop:E<1-e-E-m}, to S\'ebastien Boucksom for clarifying some points in \cite{BBJ18} and also to M. Hattori for pointing out an imprecision in the statement of Theorem \ref{thm:cscK} in previous versions. The author is supported by the China post-doctoral grant BX20190014 and NSFC grant 12101052.

\section{Setup and the main results}
\label{sec:setup-results}

\subsection{Notation and definitions}

Let $X$ be a projective manifold of dimension $n$ with 
an ample $\RR$-line bundle $L$ over it.
 Fix a smooth Hermitian metric $h$ on $L$ such that 
$$\omega:=-\ddc\log h\in c_1(L)$$
is a K\"ahler form (here $\ddc=\frac{\sqrt{-1}\partial\bar{\partial}}{2\pi}$).
Put $V:=\int_X\omega^n=L^n.$ To make our result a bit more general, we will also fix (following \cite{BBJ18})
$$
\text{a closed positive $(1,1)$-current $\theta$ with klt singularities,}
$$
meaning that, when writing $\theta=\ddc\psi$ locally, one has $e^{-\psi}\in L^p_{loc}$ for some $p>1$. A case of particular interest is when $\theta=[\Delta]$ is the integration current along some effective klt divisor $\Delta$, which  relates to the edge-cone metrics for log pairs. The reader may take $\theta=0$ for simplicity as it will make no essential difference.

Now we recall the definition of $\delta$-invariant, which was first introduced by Fujita--Odaka \cite{FO18} using basis type divisors, and then reformulated by Blum--Jonsson \cite{BJ17} in a more valuative fashion. To incorporate $\theta$, we will use the following definition of Berman--Boucksom--Jonsson \cite{BBJ18}:
\begin{equation*}
    \label{eq:def-delta-L}
    \delta(L;\theta):=\inf_{E}\frac{A_\theta(E)}{S_L(E)}.
\end{equation*}
Here $E$ runs through all the prime divisors \emph{over} $X$, i.e., $E$ is a divisor contained in some birational model $Y\xrightarrow{\pi}X$ over $X$.
Moreover,
$$
A_\theta(E):=1+\ord_E(K_Y-\pi^*K_X)-\ord_E(\theta)
$$
denotes the log discrepancy, where $\ord_E(\theta)$
is the Lelong number of $\pi^*\theta$ at a very generic point of $E$. And
$$
S_L(E):=\frac{1}{\vol(L)}\int_0^\infty\vol(\pi^*L-xE)dx$$ 
denotes the expected vanishing order of $L$ along $E$. 

Historically, the case of the most interest is when $L=-K_X$ and $\theta=0$, i.e., the Fano case.
Regarding the existence of K\"ahler--Einstein metrics on such manifolds, a notion called K-stability was introduced by Tian \cite{Tian97} and later reformulated more algebraically by Donaldson \cite{Don02}. This stability notion has recently been further polished by Fujita and Li's valuative criterion \cite{Li17,Fuj19}, and we now know (see \cite[Theorem B]{BJ17}) that $\delta(-K_X)>1$ is equivalent to $(X,-K_X)$ being uniformly K-stable, a condition stronger than K-stability (but actually these two are equivalent, at least in the smooth setting). It is also known that uniform K-stability is equivalent to the uniform Ding stability of Berman \cite{Ber16}. More recently Boucksom--Jonsson \cite{BJ17} further extend the definition of uniform Ding stability to general polarizations using $\delta$-invariants, which we will adopt in this article.

\begin{definition}
We say $(X,L, \theta)$ is uniformly Ding stable if $\delta(L;\theta)>1$.
\end{definition}

Under the YTD framework, it is expected that such a notion would imply the existence of tKE metrics.
In the literature, the most examined case is when $c_1(L)=c_1(X)-[\theta]$, namely, the ``log Fano'' setting.  
By using continuity methods (cf. \cite{Tian15,CDS15,DS16YTD,LTW17,TW19}) or the variational approach (cf. \cite{BBJ18,LTW19,Li19-G-uniform}), we now have a fairly good understanding of the YTD conjecture in this scenario.
 The upshot is that one can indeed find a K\"ahler current $\omega_{tKE}\in c_1(L)$ solving
\begin{equation*}
	\label{eq:tKE-eq}
	\Ric(\omega_{tKE})=\omega_{tKE}+\theta
\end{equation*}
under the stability assumption.
Here $\Ric(\cdot):=-\ddc\log\det(\cdot)$ denotes the Ricci operator. The solution $\omega_{tKE}$ is precisely what we mean by a twisted K\"ahler--Eisntein metric (cf. also \cite{BBEGZ19,BBJ18}).

However, to the author's knowledge, all the known approaches to the above statement does not work well in the case where $\theta$ is merely quasi-positive, one main difficulty being that there is no convexity available for twisted K-energy in the non-Fano setting.
In what follows we will present a quantization approach to circumvent this difficulty, which allows us to work even without the Fano condition. 

More precisely, given any (not necessarily semi-positive) smooth representative $\eta\in c_1(X)-c_1(L)-[\theta]$, we want to investigate the following tKE equation:
\begin{equation}
	\label{eq:tKE-general}
	\Ric(\omega_{tKE})=\omega_{tKE}+\eta+\theta.
\end{equation}
To study this, 
a crucial input is taken from the work of Ding \cite{Ding88}, who essentially showed that the solvability of the above equation is governed by certain Moser--Trudinger type inequality.
 Inspired by this viewpoint, the author introduced an analytic $\delta$-invariant in \cite{Zha20}, which we now turn to describe.

Put
$$
\cH(X,\omega):=\big\{\phi\in C^\infty(X,\RR)\big|\omega_\phi:=\omega+\ddc\phi>0\big\}.
$$
Let $E:\cH(X,\omega)\rightarrow\RR$ denote the Monge--Amp\`ere energy defined by
\begin{equation*}
E(\phi):=
    \frac{1}{(n+1)V}\sum_{i=0}^n\int_X\phi\omega^{n-i}\wedge\omega^i_\phi\ \text{for }\phi\in\cH(X,\omega).
\end{equation*}
Also fix a smooth representative $\theta_0\in[\theta]$, so we can write $\theta=\theta_0+\ddc\psi$ for some usc function $\psi$ on $X$. We may rescale $\psi$ such that
\begin{equation}
	\label{eq:def-mu-theta}
	\mu_\theta:=e^{-\psi}\omega^n
\end{equation}
defines a probability measure on $X$ (i.e., $\int_Xd\mu_\theta=1$). Note that $\theta$ being klt is equivalent to saying that for any $p>1$, sufficiently close to 1, there exists $A_p>0$ such that
\begin{equation}
\label{eq:psi-L-p}
	\int_Xe^{-p\psi}\omega^n<A_p.
\end{equation}
The analytic $\delta$-invariant of $(X,L,\theta)$ is then defined by
\begin{equation}
	\label{eq:def-delta-A}
	\delta^A(L;\theta):=\sup\bigg\{\lambda>0\bigg|\exists C_\lambda>0\text{ s.t. }\int_Xe^{-\lambda(\phi-E(\phi))}d\mu_\theta<C_\lambda\text{ for any }\phi\in\cH(X,\omega)\bigg\},
\end{equation}
which does not depend on the choice of $\omega$ or $\theta_0$. As explained in \cite{Zha20}, $\delta^A(L;\theta)>1$ is equivalent the coercivity of certain twisted Ding functional whose critical point gives rise to the desired tKE metric.
It is further conjectured in \cite{Zha20} that one should have 
$\delta(L;\theta)=\delta^A(L;\theta).$
Given this, then \eqref{eq:tKE-general} can be solved when $\delta(L;\theta)>1$, i.e., when $(X,L,\theta)$ is uniformly Ding stable.
 
\subsection{Main results}

In this article we confirm the aforementioned conjecture.
\begin{theorem}[Main Theorem]
\label{thm:delta-delta-A}
	For any ample $\RR$-line bundle $L$, one has
	$$
	\delta(L;\theta)=\delta^A(L;\theta).
	$$
\end{theorem}

In particular uniform Ding stability implies the coercivity of twisted Ding functionals and 
as a consequence, we obtain a new proof of the uniform YTD conjecture and generalize the known results in the log Fano case (e.g., \cite[Theorem A]{BBJ18}) to the following more general setting, with possibly irrational polarizations.

\begin{theorem}
\label{thm:YTD}
Assume that $(X,L,\theta)$ is uniformly Ding stable. Then for any  smooth form $\eta\in(c_1(X)-c_1(L)-[\theta])$, there exists a K\"ahler current $\omega_{tKE}\in c_1(L)$ solving
$$
\Ric(\omega_{tKE})=\omega_{tKE}+\eta+\theta.
$$

\end{theorem}

As mentioned in Introduction, the proof of Theorem \ref{thm:delta-delta-A} uses the quantization approach initiated in \cite{RTZ20}, which already implies one direction: $\delta^A(L;\theta)\leq\delta(L;\theta)$ when $L$ is an ample $\QQ$-line bundle. For completeness we will recall its proof in Section \ref{sec:proof}. For the other direction, $\delta^A(L;\theta)\geq\delta(L;\theta)$, we will crucially use a quantized maximum principle due to Berndtsson \cite{Bern09}, which enables us to bound $\delta^A$ from below using finite dimensional data, hence the result. The general case of $\RR$-line bundle then follows by invoking the continuity of $\delta$ and $\delta^A$ in the ample cone (cf. \cite{Zha20}). At the end of this article we will briefly explain how to generalize our approach to the coupled soliton case considered in \cite{RTZ20}.

In fact we expect that our approach can be generalized to the case of big line bundles, yielding new existence results for the general Monge--Amp\`ere equations considered in \cite{BEGZ10}, and answering some questions proposed in \cite[Section 6.3]{Zha20}. Another direction to pursue would be to consider the the case of singular varieties (as in \cite{LTW17,LTW19}) or the equivariant case (as in \cite{Li19-G-uniform}).

Now take $\theta=0$, in which case we will drop $\theta$ from our notation. Then 
Theorem \ref{thm:delta-delta-A} has the following interesting application, yielding a new criterion for the existence of cscK metrics. This also answers \cite[Question 6.13]{Zha20}.

\begin{theorem}
\label{thm:cscK}
Let $L$ be an ample $\RR$-line bundle.
Assume that $K_X+\delta(L)L$ is ample and $\delta(L)>n\mu(L)-(n-1)s(L)$, where $\mu(L):=\frac{-K_X\cdot L^{n-1}}{L^n}$ and $s(L):=\sup\{s\in\RR|-K_X-sL>0\}$. Then $X$ admits a unique constant scalar curvature K\"ahler (cscK) metric in $c_1(L)$.
\end{theorem}

Recent progress made by Ahmadinezhad--Zhuang \cite{AZ20} shows that one can effectively compute $\delta$-invariants by induction and inversion of adjunction. So we expect that Theorem \ref{thm:cscK} can be applied to find more new examples of cscK manifolds. Also observe that the assumption in Theorem \ref{thm:cscK} is purely algebraic, so the author wonders if one can show uniform K-stability for $(X,L)$ under the same condition using only algebraic argument; see \cite{DerEv20} for related discussions.

\section{Existence of canonical metrics}
\label{sec:existence}

In this section we explain how is $\delta^A$ related to the canonical metrics in K\"ahler geometry, following \cite{Zha20}. The discussions below in fact hold for general K\"ahler classes as well. 

We begin by introducing a twisted version of the $\alpha$-invariant of Tian \cite{Tian87}.
Set
\begin{equation}
	\label{eq:def-alpha}
\alpha(L;\theta):=\sup\bigg\{\alpha>0\bigg|\exists C_\alpha>0\text{ s.t. }\int_Xe^{-\alpha(\phi-\sup\phi)}d\mu_\theta<C_\alpha\text{ for all }\phi\in\cH(X,\omega)\bigg\}.
\end{equation}

\begin{lemma}
	One always has $\alpha(L;\theta)>0$.
\end{lemma}

\begin{proof}
Using H\"older's inequality, the assertion follows from
	\cite[Proposition 2.1]{Tian87} and \eqref{eq:psi-L-p}.
\end{proof}

As a consequence, one also has
$$
\delta^A(L;\theta)>0
$$
since $E(\phi)\leq\sup\phi$.
Note that $\alpha(L;\theta)$ will be used several times in this article, as it can effectively control bad terms when doing integration.
\subsection{Twisted Ding functional}
In this part we relate $\delta^A$ to tKE metrics.
Pick any smooth representative $\eta\in c_1(X)-c_1(L)-[\theta]$.
Then we can find $f\in C^\infty(X,\RR)$ satisfying
$$
\Ric(\omega)=\omega+\eta+\theta_0+\ddc f,
$$
where we recall that $\theta_0\in[\theta]$ is the smooth representative we have fixed. Then the twisted Ding functional is defined by
$$
D_{\theta+\eta}(\phi):=-\log\int_Xe^{f-\phi}d\mu_\theta-E(\phi)\text{ for }\phi\in\cH(X,\omega).
$$
Actually one can extend $D_{\theta+\eta}(\cdot)$ to the larger space
$
\cE^1(X,\omega)
$
(see \cite{BBGZ} for the definition). Using variational argument,  a critical point  $\phi\in\cE^1(X,\omega)$  of $D_{\theta+\eta}(\cdot)$ will give rise to a solution to \eqref{eq:tKE-general} (see \cite[Section 4]{BBEGZ19}). A sufficient condition to guarantee the existence of such a critical point is called \emph{coercivity}, which we recall as follows.

\begin{definition}
	The twisted Ding functional $D_{\theta+\eta}(\cdot)$ is called coercive if there exist $\varepsilon>0$ and $C>0$ such that
	$$
	D_{\theta+\eta}(\phi)\geq\varepsilon(\sup\phi-E(\phi))-C\text{ for all }\phi\in\cH(X,\omega).
	$$
\end{definition}
Using Demailly's regularization, the above definition is equivalent the coercivity investigated in \cite{BBEGZ19} and hence $D_{\theta+\eta}$ being coercive implies the existence of a solution to \eqref{eq:tKE-general} by \cite[Section 4]{BBEGZ19}.

\begin{proposition}
	If $\delta^A(L;\theta)>1$, then $D_{\theta+\eta}(\cdot)$ is coercive for any smooth representative $\eta\in c_1(X)-c_1(L)-[\theta]$.
\end{proposition}
\begin{proof}
	This is already contained in \cite[Proposition 3.6]{Zha20} (which in fact says that the reverse direction is also true).
	It suffices to show that, for some $\varepsilon>0$ and $C>0$,
	$$
	-\log\int_Xe^{-\phi}d\mu_\theta-E(\phi)\geq\varepsilon(\sup\phi-E(\phi))-C\text{ for any }\phi\in\cH(X,\omega).
	$$
To see this, fix $\lambda\in(1,\delta^A(L;\theta))$ and $\alpha\in(0,\min\{1,\alpha(L;\theta)\})$. Then by H\"older's inequality,
\begin{equation*}
	\begin{aligned}
		-\log\int_Xe^{-\phi}d\mu_\theta-E(\phi) &\geq-\frac{1-\alpha}{\lambda-\alpha}\log\int_Xe^{-\lambda\phi}d\mu_\theta-\frac{\lambda-1}{\lambda-\alpha}\int_Xe^{-\alpha\phi}d\mu_\theta-E(\phi)\\
		&=-\frac{1-\alpha}{\lambda-\alpha}\log\int_Xe^{-\lambda(\phi-E(\phi))}d\mu_\theta-\frac{\lambda-1}{\lambda-\alpha}\int_Xe^{-\alpha(\phi-\sup\phi)}d\mu_\theta\\
		&\ \ \ \ \ +\frac{\alpha(\lambda-1)}{\lambda-\alpha}(\sup\phi-E(\phi)).
	\end{aligned}
\end{equation*}
Then the assertion follows from \eqref{eq:def-delta-A} and \eqref{eq:def-alpha}.
\end{proof}

\begin{corollary}
\label{cor:delta-A>1=>tKE}
	If $\delta^A(L;\theta)>1$, then there exists a solution to \eqref{eq:tKE-general} for any smooth representative $\eta\in c_1(X)-c_1(L)-[\theta]$.
\end{corollary}

\subsection{K-energy and constant scalar curvature metric}
In this part we relate $\delta^A$ to cscK metrics.
For simplicity assume $\theta=0$, and hence $\theta$ will be abbreviated in our notation. Let us first recall several functionals. For $\phi\in\cH(X,\omega)$, define
\begin{itemize}
	\item $I$-functional:
	$
	I(\phi):=\frac{1}{V}\int_X\phi(\omega^n-\omega_\phi^n);
	$
	\item $J$-functional:
	$
	J(\phi):=\frac{1}{V}\int_X\phi\omega^n-E(\phi);	$
	\item Entropy:
	$
H(\phi):=\frac{1}{V}\int_X\log\frac{\omega_\phi^n}{\omega^n}\omega^n_\phi;
$
	\item $\cJ$-Energy:
$
\mathcal{J}(\phi):=n\frac{(-K_X)\cdot L^{n-1}}{L^n}E(\phi)-\frac{1}{V}\int_X\phi\Ric(\omega)\wedge\sum_{i=0}^{n-1}\omega^i\wedge\omega^{n-1-i}_\phi;
$
\item K-energy:
$
K(\phi):=H(\phi)+\mathcal{J}(\phi).
$
\end{itemize}
A K\"ahler metric $\omega_\phi\in c_1(L)$ is a cscK metric if and only if $\phi$ is a critical point of the K-energy (cf. \cite{Mabuchi}).
The following result says that $\delta^A(L)$ is the coercivity threshold of $H(\phi)$.
\begin{proposition}
	\cite[Proposition 3.5]{Zha20}
	\label{prop:delta=entropy}
	We have
$$
\delta^A(L)=\sup\bigg\{\lambda>0\bigg|
    \exists\ C_\lambda>0\text{ s.t. }
    H(\phi)\geq\lambda(I-J)(\phi)-C_\lambda\text{ for all }
    \phi\in\mathcal{H}(X,\omega)
    \bigg\}.
$$
\end{proposition}

Now let $\mu(L):=\frac{-K_X\cdot L^{n-1}}{L^n}$ denote the slope and $s(L):=\sup\{s\in\RR|-K_X-sL>0\}$ the nef threshold. As explained in \cite[Section 6.2]{Zha20},
if $K_X+\delta^A(L)L$ is ample and $\delta^A(L)+(n-1)s(L)-n\mu(L)>0$, then for some $\varepsilon>0$ and $C_\varepsilon>0$,
$$
K(\phi)\geq\varepsilon(I-J)(\phi)-C_\varepsilon\text{ for all }
\phi\in\cH(X,\omega),
$$
meaning that the K-energy is coercive. So by Chen--Cheng \cite[Theorem 4.1]{CC18}, there exists a cscK metric in $c_1(L)$. Moreover by \cite[Theorem 1.3]{BerBern17} such a metric is unique as in this case the automorphism group must be discrete.
As a consequence, we have the following
  \begin{corollary}\cite[Corollary 6.12]{Zha20}
  \label{cor:delta-cscK}
Assume that $K_X+\delta^A(L) L$ is ample and
 $
 \delta^A(L)>n\mu(L)-(n-1)s(L),
 $
 then there exists a unique cscK metric in $c_1(L)$.
 \end{corollary}

\section{Quantization}
\label{sec:quantization}

We collect some necessary quantization techniques for the proof of our main theorem.
In this section we assume $L$ to be an ample line bundle over $X$. By rescaling $L$ we will assume further that $L$ is very ample.

Put
$$
R_m:=H^0(X,mL)\text{ and }d_m:=\dim R_m.
$$
As in Section \ref{sec:setup-results}, fix a smooth positively curved Hermitian metric $h$ on $L$ with $\omega:=-\ddc\log h$.

\subsection{Bergman space}
Note that there is a natural Hermitian inner product
$$
H_m:=\int_Xh^m(\cdot,\cdot)\omega^n
$$
on $R_m$ induced by $h$. More generally, for any bounded function $\phi$ on $X$, we may consider
$$
H_m^\phi:=\int_X(he^{-\phi})^m(\cdot,\cdot)\omega^n.
$$
So in particular, $H_m=H^0_m$.

Now put
$$
\cP_m(X,L):=\bigg\{\text{Hermitian inner product on } R_m\bigg\}.
$$
and
$$
\cB_m(X,\omega):=\bigg\{\phi=\frac{1}{m}\log\sum_{i=1}^{d_m}|\sigma_i|^2_{h^m}\bigg|\{\sigma_i\}\text{ is a basis of $R_m$}\bigg\}.
$$
The classical Fubini--Study map $FS:\cP_m(X,L)\rightarrow\cB_m(X,\omega)$ is a bijection, where $FS$ is defined by
$$
FS(H):=\frac{1}{m}\log\sum_{i=1}^{d_m}|\sigma_i|^2_{h^m}\text{ for }H\in\cP_m\text{ where }\{\sigma_i\}\text{ is any $H$-orthonormal basis}.
$$
In particular $\cB_m(X,\omega)\subseteq\cH(X,\omega)$ is a finite dimensional subspace (when identified with $\cP_m(X,L)\cong GL(d_m,\CC)/U(d_m)$).

For any $\phi\in\cH(X,\omega)$, we set for simplicity
$$
\phi^{(m)}:=FS(H^\phi_m).
$$
It then follows from the definition that
\begin{equation}
    \label{eq:int-m-phi--phi-m=d-m}
    \int_Xe^{m(\phi^{(m)}-\phi)}\omega^n=d_m\text{ for any }\phi\in\cH(X,\omega).
\end{equation}
This simple identity will be used in the proof of Theorem \ref{thm:delta-delta-A}.

Note that any two Hermitian inner products can be joined by the (unique) \emph{Bergman geodesic}. More specifically, given any two $H_{m,0}, H_{m,1}\in\cP_m(X,L)$, one can find an $H_{m,0}$-orthonormal basis under which $H_{m,1}=\mathrm{diag}(e^{\mu_1},...,e^{\mu_{d_m}})$ is diagonalized. Then the Bergman geodesic $H_t$ takes the form
$$
H_{m,t}:=\mathrm{diag}(e^{\mu_1t},...,e^{\mu_{d_m}t}).
$$

\subsection{Quantized $\delta$-invariant}

Now as in \cite{RTZ20}, we consider the following \emph{quantized Monge--Amp\`ere energy}:
$$
E_m(\phi):=\frac{1}{md_m}\log\frac{\det H_m}{\det FS^{-1}(\phi)} \text{ for }\phi\in\cB_m(X,\omega).
$$
In the literature this is also known as (up to a sign) Donaldson's $\cL_m$-functional (cf. \cite{D05}).
Observe that $E_m(FS(\cdot))$ is linear along any Bergman geodesics emanating from $H_m$. So in particular
\begin{equation}
	\label{eq:E_m=d/dt}
	E_m(FS(H_{m,1}))=\frac{d}{dt}\bigg|_{t=0}E_m(FS(H_{m,t}))
\end{equation}
for any Bergman geodesic $[0,1]\ni t\mapsto H_{m,t}$ with $H_{m,0}=H_m.$

Put
\begin{equation}
    \label{eq:def-delta-m}
    \delta_m(L;\theta):=\sup\bigg\{\lambda>0\bigg|\exists C_\lambda>0\text{ s.t. }\int_Xe^{-\lambda(\phi-E_m(\phi))}d\mu_\theta<C_\lambda\text{ for any }\phi\in\cB_m\bigg\}.
\end{equation}
By our previous work \cite[Theorem B.3]{RTZ20} (whose proof requires the estimate of Demailly--Koll\'ar \cite{DK01}), this coincides with the original basis divisor formulation of Fujita--Odaka \cite{FO18}. Moreover, by \cite[Theorem A]{BJ17} and \cite[Theorem 7.3]{BBJ18} the limit of $\delta_m(L;\theta)$ exists and one has
\begin{equation}
    \label{eq:delta=lim-delta-m}
    \delta(L;\theta)=\lim_{m\rightarrow\infty}\delta_m(L;\theta).
\end{equation}
Note that $\delta_m(L;\theta)$ characterizes the coercivity of certain quantized Ding functional, whose critical points correspond to ``balanced metrics"; see \cite[Theorem B.7]{RTZ20} for a quantized version of Theorem \ref{thm:YTD}.

\subsection{Comparing $E$ with $E_m$}

Given any $\phi\in\cH(X,\omega)$, it has been known since the work of Donaldson that
$
E(\phi)=\lim_{m\rightarrow\infty}E_m(\phi^{(m)}).
$
But this convergence is not uniform when $\phi$ varies in $\cH(X,\omega)$, which is the main stumbling block in the quantization approach. To overcome this, we recall a quantized maximum principle due to Berndtsson \cite{Bern09}.

The setup goes as follows. For any ample line bundle $E$ over $X$, let $g$ be a smooth positively curved metric on $E$ with $\eta:=-\ddc\log g>0$ being its curvature form. Pick two elements $\phi_0,\phi_1\in\cH(X,\eta)$. It was shown by Chen \cite{Chen00} and more recently by Chu--Tosatti--Weinkove \cite{CTW17} that there always exists a \emph{$C^{1,1}$-geodesic} $\phi_t$ joining $\phi_0$ and $\phi_1$.
For the reader's convenience, we briefly recall the definition. Let $[0,1]\ni t\mapsto\phi_t$ be a family of functions on $[0,1]\times X$ with $C^{1,1}$ regularity up to the boundary. Let $S:=\{0<\mathrm{Re}\,s<1\}\subset\CC$ be the unit strip and let $\pi:S\times X\rightarrow X$ denote the projection to the second component. Then we say $\phi_t$ is a \emph{$C^{1,1}$-subgeodesic} if it satisfies
$
\pi^*\eta+\ddc_{S\times X}\phi_{\mathrm{ Re }\,s}\geq0.
$
We say it is a $C^{1,1}$-geodesic if it further satisfies the homogenous Monge--Amp\`ere equation:
$
\big(\pi^*\eta+\ddc_{S\times X}\phi_{\mathrm{ Re }\,s}\big)^{n+1}=0.
$

Now given any $C^{1,1}$ subgeodesic joining $\phi_0$ and $\phi_1$, one may consider
$$
Hilb^{\phi_t}:=\int_Xg(\cdot,\cdot)e^{-\phi_t},
$$
which is a family of Hermitian inner products on $H^0(X,E+K_X)$ joining $Hilb^{\phi_0}$ and $Hilb^{\phi_1}$. Note that we do not need any volume form in the above integral.
Then Berndtsson's quantized maximum principle says the following, which in fact holds for subgeodesics with much less regularity; see \cite[Proposition 2.12]{DLR19}.
\begin{proposition}\cite[Proposition 3.1]{Bern09}
\label{prop:Berndtsson-max-principle}
	Let $[0,1]\ni t\mapsto H_t$ be the Bergman geodesic connecting $Hilb^{\phi_0}$ and $Hilb^{\phi_1}$. Then one has
	$$
	H_t\leq Hilb^{\phi_t} \text{ for }t\in[0,1].
	$$
\end{proposition}

We will now apply this result to the setting where $E:=mL-K_X$ and $g:=h^m\otimes\omega^n$. As a consequence, we obtain the following key estimate, which can be viewed as a weak version of the ``partial $C^0$ estimate''.

\begin{proposition}
\label{prop:E<1-e-E-m}
For any $\varepsilon\in(0,1)$, there exist $m_0=m_0(X,L,\omega,\varepsilon)\in\NN$ such that
$$
E(\phi)\leq E_m\big(((1-\varepsilon)\phi)^{(m)}\big)+\varepsilon\sup\phi
$$
for any $m\geq m_0$ and any $\phi\in\cH(X,\omega)$.
\end{proposition}

\begin{proof}
Since the statement is translation invariant, we assume that $\sup\phi=0$.
Let $[0,1]\ni t\mapsto\phi_t$ be a $C^{1,1}$ geodesic connecting $0$ and $\phi$, with $\phi_0=0$ and $\phi_1=\phi$. The geodesic condition implies that $\phi_t$ is convex in $t$ so we have
$$
\dot\phi_0:=\frac{d}{dt}\bigg|_{t=0}\phi_t\leq0
$$
as $\phi\leq0$. 
Put
$
\tilde{\phi}_t:=(1-\varepsilon)\phi_t.
$
Observe that $(he^{-\tilde{\phi}_t})^m\otimes\omega^n$ gives rise to a family of Hermitian metrics on $mL-K_X$, which is in fact a $C^{1,1}$ subgeodesic whenever $m$ satisfies $m \varepsilon\omega\geq-\Ric(\omega)$.  Indeed, let $S:=\{0<\mathrm{Re}\,s<1\}\subset\CC$ be the unit strip and let $\pi:S\times X\rightarrow X$ denote the projection to the second component. Then $(he^{-\tilde{\phi}_{\mathrm{Re}\,s}})^m\otimes\omega^n$ induces a Hermitian metric on $\pi^*(mL-K_X)$ over $S\times X$ whose curvature form satisfies
$$
\pi^*(m\omega+\Ric(\omega))+m(1-\varepsilon)\ddc_{S\times X}\phi_{\mathrm{Re}\,s}\geq0
$$
whenever $m \varepsilon\omega\geq-\Ric(\omega)$.
It then follows from Proposition \ref{prop:Berndtsson-max-principle} that
$$
H_{m,t}\leq H_m^{\tilde{\phi}_t} \text{ for }t\in[0,1],
$$
where $[0,1]\ni t\mapsto H_{m,t}$ is the Bergman geodesic in $\cP_m(X,L)$ joining $H^0_m$ and $H^{(1-\varepsilon)\phi}_m$ with $H_{m,0}=H^0_m$ and $H_{m,1}=H^{(1-\varepsilon)\phi}_m$.
So we obtain that
$$
E_m(FS(H_{m,t}))\geq E_m(FS(H_m^{\tilde{\phi}_t}))\text{ for }t\in[0,1],
$$
with equality at $t=0,1$. Fixing an $H^0_m$-orthonormal basis $\{s_i\}$ of $R_m$, then by \eqref{eq:E_m=d/dt} we obtain that
\begin{equation*}
    \begin{aligned}
    E_m\big(((1-\varepsilon)\phi)^{(m)}\big)&=\frac{d}{dt}\bigg|_{t=0}E_m(FS(H_{m,t}))
    \geq\frac{d}{dt}\bigg|_{t=0}E_m(FS(H_{m}^{\tilde{\phi}_t}))
    =\frac{1-\varepsilon}{d_m}\int_X\dot\phi_0\bigg(\sum_{i=1}^{d_m}|s_i|^2_{h^m}\bigg)\omega^n,
    \end{aligned}
\end{equation*}
where the last equality is from a direct calculation using the definition of $E_m$. 
Now by the first order expansion of Bergman kernels going back to Tian \cite{Tian89} (with respect to the background metric $\omega$), one has
$$
\frac{\sum_{i=1}^{d_m}|s_i|^2_{h^m}}{d_m}\leq\frac{1}{(1-\varepsilon)V}
$$
for all $m\gg 1$. So we arrive at (recall $\dot\phi_0\leq0$)
$$
E_m\big(((1-\varepsilon)\phi)^{(m)}\big)\geq\frac{1}{V}\int_X\dot\phi_0\omega^n=E(\phi),
$$
where the last equality follows from the well-known fact that $E$ is linear along the geodesic $\phi_t$. This completes proof.
\end{proof}

\begin{remark}
\rm{
	After the appearance of this work on arXiv, the author was informed by Berndtsson that Proposition \ref{prop:E<1-e-E-m} also follows from the fact that $E_m(FS(H^{\tilde{\phi}_t}_m))$ is convex in t. And Berman kindly communicated to the author that, using Berndtsson's convexity, our estimate is essentially contained in \cite{BF14}; see in particular (3.4) in loc. cit. The author is very grateful to them for communications! But we need to emphasize that our proof here is slightly different, with a small advantage that it can be directly generalized to the weighted setting to treat soliton type metrics; see also Remark \ref{rmk:coupled-case}.}
\end{remark}

One can also bound $E$ from below in terms of $E_m$ on the Bergman space $\cB_m(X,\omega)$. This direction is already known; see \cite[Lemma 7.7]{BBGZ} or \cite[Lemma 5.2]{RTZ20}. We record it here for completeness. 
\begin{proposition}
\label{prop:E-m<E}
	For any $\varepsilon>0$, there exists $m_0=m_0(X,L,\omega,\varepsilon)\in\NN$ such that
	$$
	E_m(\phi)\leq(1-\varepsilon)E(\phi)+\varepsilon\sup\phi+\varepsilon.
	$$
	for any $m\geq m_0$ and $\phi\in\cB_m(X,\omega)$.
\end{proposition}

\section{Proving $\delta=\delta^A$}
\label{sec:proof}

In this section we prove our main results. Firstly, we prove Theorem \ref{thm:delta-delta-A} in the case where $L$ is a \emph{bona fide} ample line bundle, so that we can apply quantization techniques.

\begin{theorem}
\label{thm:delta=delta-A'}
Let $L$ be an ample line bundle, then one has
$
\delta^A(L;\theta)=\delta(L;\theta)
$
\end{theorem}

\begin{proof}
The proof splits into two steps.

\textbf{Step 1: $\delta^A(L;\theta)\leq\delta(L;\theta)$}. 

In the view of \eqref{eq:delta=lim-delta-m}, it suffices to show that, for any $\lambda\in(0,\delta^A(L;\theta))$ one has $\delta_m(L;\theta)>\lambda$ for all $m\gg1$. In other words, for any $m\gg1$, we need to find some constant $C_{m,\lambda}>0$ such that
$$
\int_Xe^{-\lambda(\phi-E_m(\phi))}d\mu_\theta<C_{m,\lambda}\text{ for all }\phi\in\cB_m(X,\omega).
$$
Assume that $\sup\phi=0$. For any small $\varepsilon>0$, by Proposition \ref{prop:E-m<E} and H\"older's inequality,
\begin{equation*}
	\begin{aligned}
		\int_Xe^{-\lambda(\phi-E_m(\phi))}d\mu_\theta&\leq\int_Xe^{-\lambda(\phi-(1-\varepsilon)E(\phi))+\lambda\varepsilon}d\mu_\theta
		=e^{\lambda\varepsilon}\cdot\int_Xe^{-\lambda(1-\varepsilon)(\phi-E(\phi))}\cdot e^{-\lambda\varepsilon\phi} d\mu_\theta\\
		&\leq e^{\lambda\varepsilon}\bigg(\int_Xe^{\frac{-\lambda(1-\varepsilon)}{1-\frac{\lambda\varepsilon}{\alpha}}(\phi-E(\phi))}d\mu_\theta\bigg)^{1-\frac{\lambda\varepsilon}{\alpha}}\bigg(\int_Xe^{-\alpha\phi}d\mu_\theta\bigg)^{\frac{\lambda\varepsilon}{\alpha}}
	\end{aligned}
\end{equation*}
holds for all $m\geq m_0(X,L,\omega,\varepsilon)$, where $\alpha\in(0,\alpha(L;\theta))$ is some fixed number.
We may fix $\varepsilon\ll1$ such that
$$
\frac{\lambda(1-\varepsilon)}{1-\frac{\lambda\varepsilon}{\alpha}}<\delta^A(L;\theta).
$$
Then by \eqref{eq:def-delta-A} and \eqref{eq:def-alpha}, there exist $C_\lambda>0$ and $C_\alpha>0$ such that
$$
\int_Xe^{-\lambda(\phi-E_m(\phi))}d\mu_\theta<e^{\lambda\varepsilon}(C_\lambda)^{1-\frac{\lambda\varepsilon}{\alpha}}(C_\alpha)^{\frac{\lambda\varepsilon}{\alpha}}
$$
for all $\phi\in\cB_m(X,\omega)$ whenever $m$ is large enough. This proves the assertion.

\textbf{Step 2: $\delta^A(L;\theta)\geq\delta(L;\theta)$.}

It suffices to show that,
for any $\lambda\in(0,\delta(L;\theta))$, there exists $C_\lambda>0$ such that
$$
\int_Xe^{-\lambda(\phi-E(\phi))}d\mu_\theta< C_\lambda\text{ for any $\phi\in\cH(X,\omega)$.}
$$
Again assume that $\sup\phi=0$. Fix any number $\alpha\in(0,\alpha(L;\theta))$. Fix $p_0>1$ such that \eqref{eq:psi-L-p} holds for any $p\in(1,p_0)$.
Let also $\varepsilon>0$ be a sufficiently small number, to be fixed later. Set 
$\tilde{\phi}:=(1-\varepsilon)\phi.$
Then by Proposition \ref{prop:E<1-e-E-m} and the generalized H\"older inequality, for any $m\geq m_0(X,L,\omega,\varepsilon)$, we can write
\begin{equation*}
    \begin{aligned}
    \int_Xe^{-\lambda\big(\phi-E(\phi)\big)}d\mu_\theta&\leq\int_Xe^{-\lambda\big(\phi-E_m(\tilde{\phi}^{(m)})\big)}d\mu_\theta
    = \int_Xe^{\lambda\big(\tilde{\phi}^{(m)}-\tilde{\phi}\big)}\cdot e^{-\lambda\big(\tilde{\phi}^{(m)}-E_m(\tilde{\phi}^{(m)})\big)}\cdot e^{-\lambda\varepsilon\phi}d\mu_\theta\\
    &\leq\bigg(\int_Xe^{\sqrt{m}(\tilde{\phi}^{(m)}-\tilde{\phi})}d\mu_\theta\bigg)^{\frac{\lambda}{\sqrt{m}}}\bigg(\int_Xe^{\frac{-\lambda\big(\tilde{\phi}^{(m)}-E_m(\tilde{\phi}^{(m)})\big)}{1-\frac{\lambda}{\sqrt{m}}-\frac{\lambda\varepsilon}{\alpha}}}d\mu_\theta\bigg)^{1-\frac{\lambda}{\sqrt{m}}-\frac{\lambda\varepsilon}{\alpha}}\bigg(\int_Xe^{-\alpha\phi}d\mu_\theta\bigg)^{\frac{\lambda\varepsilon}{\alpha}}\\
    &\leq (d_m)^{\frac{\lambda}{m}}\bigg(\int_Xe^{-\frac{\sqrt{m}\psi}{\sqrt{m}-1}}\omega^n\bigg)^{\frac{\lambda}{\sqrt{m}}-\frac{\lambda}{m}}\bigg(\int_Xe^{\frac{-\lambda\big(\tilde{\phi}^{(m)}-E_m(\tilde{\phi}^{(m)})\big)}{1-\frac{\lambda}{\sqrt{m}}-\frac{\lambda\varepsilon}{\alpha}}}d\mu_\theta\bigg)^{1-\frac{\lambda}{\sqrt{m}}-\frac{\lambda\varepsilon}{\alpha}}\bigg(\int_Xe^{-\alpha\phi}d\mu_\theta\bigg)^{\frac{\lambda\varepsilon}{\alpha}},\\
    \end{aligned}
\end{equation*}
where we used \eqref{eq:def-mu-theta} and \eqref{eq:int-m-phi--phi-m=d-m} in the last inequality.
We now fix $\varepsilon\ll1$ and $m\gg m_0(X,L,\omega,\varepsilon)$ such that
$$
\frac{\sqrt{m}}{\sqrt{m}-1}<p_0\text{ and }\frac{\lambda}{1-\frac{\lambda}{\sqrt{m}}-\frac{\lambda\varepsilon}{\alpha}}<\delta_m(L;\theta).
$$
Then by \eqref{eq:psi-L-p}, \eqref{eq:def-delta-m} and \eqref{eq:def-alpha} there exist $A_m>0$, $C_{m,\lambda}>0$ and $C_\alpha>0$ (recall $\sup\phi=0$) such that
\begin{equation*}
    \begin{aligned}
    \int_Xe^{-\lambda(\phi-E(\phi))}d\mu_\theta&< (d_m)^{\frac{\lambda}{m}}\cdot(A_m)^{\frac{\lambda}{\sqrt{m}}-\frac{\lambda}{m}}\cdot (C_{m,\lambda})^{1-\frac{\lambda}{\sqrt{m}}-\frac{\lambda\varepsilon}{\alpha}}\cdot(C_\alpha)^{\frac{\lambda\varepsilon}{\alpha}}.\\
    \end{aligned}
\end{equation*}
Note that all the constants are uniform, independent of $\phi$. So we finally arrive at $\int_Xe^{-\lambda(\phi-E(\phi))}d\mu_\theta<C_\lambda$ for some uniform $C_\lambda>0$, as desired.

\end{proof}

\begin{proof}
	[Proof of Theorem \ref{thm:delta-delta-A}]
	Since $\delta(L;\theta)=\delta^A(L;\theta)$ holds for any ample line bundle, by rescaling, it holds for any ample $\QQ$-line bundle. Now by the continuity of $\delta$ and $\delta^A$ in the ample cone (cf. \cite{Zha20}), the same assertion holds for any ample $\RR$-line bundle.
\end{proof}

\begin{proof}
	[Proof of Theorem \ref{thm:YTD}] The result follows from Theorem \ref{thm:delta-delta-A} and Corollary \ref{cor:delta-A>1=>tKE}.
\end{proof}

\begin{proof}
	[Proof of Theorem \ref{thm:cscK}]The result follows from Theorem \ref{thm:delta-delta-A} and Corollary \ref{cor:delta-cscK}.
\end{proof}

By Proposition \ref{prop:delta=entropy} we also obtain an algebraic characterization of the coercivity threshold of  the entropy. One should compare this with the non-Archimedean formulation \cite[(2.9)]{BoJ18} proposed by Berman.
\begin{corollary}
	For any ample $\RR$-line bundle $L$ one has
	$$
\delta(L)=\sup\bigg\{\lambda>0\bigg|
    \exists\ C_\lambda>0\text{ s.t. }
    H(\phi)\geq\lambda(I-J)(\phi)-C_\lambda\text{ for all }
    \phi\in\mathcal{H}(X,\omega)
    \bigg\}.
$$
\end{corollary}

\begin{remark}
\label{rmk:coupled-case}\rm{
	Finally we explain how to generalize our approach to the coupled KE/soliton case considered in \cite{RTZ20}, which then yields a uniform YTD theorem for the existence of coupled KE/soliton metrics. The extension to the coupled KE case is straightforward: one only needs to replace $\phi$ and $E(\phi)$ by $\sum_i\phi_i$ and $\sum_i E_{\omega_i}(\phi_i)$ respectively, and then slightly adjust the proof of Theorem \ref{thm:delta=delta-A'}. For the more general coupled soliton case, essentially one only needs to replace $E$ by its ``$g$-weighted'' version, $E^g$, and then adjust Propositions \ref{prop:E<1-e-E-m} and \ref{prop:E-m<E} accordingly, which can be done with the help of \cite[Proposition 4.4]{BWN14}, the asymptotics for weighted Bergman kernels. Then the argument goes through almost verbatim. See our previous work \cite{RTZ20} for more explanations. The details are left to the interested reader.}
\end{remark}

\bibliography{ref.bib}
\bibliographystyle{abbrv}

\end{document}